\documentclass[11pt,psamsfonts,twoside]{article}
\usepackage{mathrsfs}
\usepackage{bbm}
\usepackage{amsmath}
\usepackage{amsmath,amsthm}
\usepackage{amssymb,amscd}
\usepackage{amsfonts,amsbsy}
\usepackage{fancyhdr,graphicx}
\usepackage[dvips]{psfrag}
\usepackage{indentfirst}

\numberwithin{equation}{section}
\newtheorem {proposition} {Proposition}[section]
\newtheorem {theorem}     [proposition]{Theorem}

\newtheorem {lemma}       [proposition]{Lemma}

\allowdisplaybreaks

\begin{document}
\setlength{\parindent}{4ex} \setlength{\parskip}{1ex}
\setlength{\oddsidemargin}{12mm} \setlength{\evensidemargin}{9mm}
\title{Global Well-posedness of the Parabolic-parabolic Keller- Segel Model in $L^{1}(\mathbb{R}^2)\times\!{L}^{\!\infty}(\mathbb{R}^2)$ and
$H^1_b(\mathbb{R}^2)\times\!{H}^1(\mathbb{R}^2)$}
\author{Chao Deng,\quad \quad Congming Li}

\date{}
\maketitle
\begin{abstract}
 In this paper, we study global well-posedness of the two-dimensional
 Keller-Segel model in Lebesgue space and Sobolev
 space. Recall that in the paper
 ``{\it Existence and uniqueness theorem on mild solutions to the Keller-Segel system in the scaling invariant space}, \emph{J. Differential Equations},
 {252}\,(2012), 1213--1228", Kozono, Sugiyama\,\&\,Wachi studied global well-posedness of $n$($\ge3$) dimensional Keller-Segel system and posted a question
 about the even local in time existence for the Keller-Segel system with
 $L^1(\mathbb{R}^2)\times{L}^\infty(\mathbb{R}^2)$ initial data. Here we
 give an affirmative answer to this question: in fact, we show the global in
 time existence and uniqueness for $L^1(\mathbb{R}^2)\times{L}^{\!\infty}(\mathbb{R}^2)$ initial data.
 Furthermore, we prove that for any
 $H^1_b(\mathbb{R}^2)\times{H}^1(\mathbb{R}^2)$ initial data with
 $H^1_b(\mathbb{R}^2):=H^1(\mathbb{R}^2)\cap{L}^\infty(\mathbb{R}^2)$,
 there also exists a unique global mild solution to the
 parabolic-parabolic Keller-Segel model. The estimates of
 ${\sup_{t>0}}t^{1-\frac{n}{p}}\|u\|_{L^p}$ for
 $(n,p)=(2,\infty)$ and the introduced special half norm, i.e. $\sup_{t>0}t^{\frac{1}{2}}(1\!+t)^{\!-\frac{1}{2}}\|\nabla{v}\|_{L^\infty}$, are crucial in our proof.
\end{abstract}
 {\small \noindent{\bf Keywords:}\, Keller-Segel model; Fourier transformation; well-posedness; decay property; parabolic-parabolic system.\\
\noindent{\bf Mathematics Subject Classification: \,}  92C17; 35K55.}
\maketitle
\section{Introduction}
 In this article, we study the following  two-dimensional (2D) Keller-Segel model:
\begin{align}
       &{{u}_t} - \Delta{u} + \nabla\cdot({u}\nabla{v})=0     &\text{in}&\hskip.23cm(0,\infty)\times\mathbb{R}^2,\label{k-s-1}\\
       &{{v}_t} - \Delta{v} + {v} - {u}=0                     &\text{in}&\hskip.23cm(0,\infty)\times\mathbb{R}^2,\label{k-s-2}\\
       &(u,v)|_{t=0}=(u_0,v_0)                                &\text{in}&\hskip.23cm\mathbb{R}^2,\label{k-s-3}
  \end{align}
 where $(t,x)\in(0,\infty)\times\mathbb{R}^2$, $u = u(t,x)$ and $v = v(t,x)$ are the scalar valued density of amoebae and
 the scalar valued concentration of chemical attractant, respectively, while $(u_0,v_0)$
 is the given initial data. For the derivation of the equation, we
 refer to Childress and Percus \cite{ChildressPercus:1981217} and Keller and Segel
 \cite{KellerSegel:1970399}. 

  Noticing that \eqref{k-s-1}--\eqref{k-s-2} is ``almost" scale invariant since ${{u}_t} - \Delta{u} + \nabla\cdot({u}\nabla{v})=0$
  and ${{v}_t} - \Delta{v}  - {u}=0$ are invariant under the following transformations
  $$(u(t,x),v(t,x))\!\rightarrow\! (\lambda^2 u(\lambda^2t,\lambda x),v(\lambda^2t,\lambda{x})) \quad\quad\text{for }\lambda>0.$$
  The idea of using a functional setting invariant by scaling is now classical
  and originates several works, see for instance, global existence of mild solutions to system \eqref{k-s-1}--\eqref{k-s-3} for initial
   $(u_0,v_0)\!\in\!H^{\frac{n}{r}-2,r}(\mathbb{R}^n)\times{H}^{\frac{n}{r},r}(\mathbb{R}^n)$
  with $\max\{1,\frac{n}{4}\}<r<\frac{n}{2}$ in \cite{KozonoSugiyama:20091}, for initial
   $(u_0,v_0)\in\!L^{\!{n}/{2}}_w(\mathbb{R}^n)\times${\small {\it BMO}}$(\mathbb{R}^n)$
  with $n\ge 3$ in \cite{KozonoSugiyama:20081467}, and for initial
  $(u_0,v_0)\in\!L^{\!\frac{n}{2}}(\mathbb{R}^n)\times\dot{H}^{2\alpha,\frac{n}{2\alpha}}(\mathbb{R}^n)$ with $n\ge3$ and
  $\frac{n}{2(n+2)}<\alpha\le\frac{1}{2}$ in \cite{KozonoSugiyamaWachi:20121213}.
  It is also known that apart from existence and uniqueness of mild solutions in scale invariant spaces,
  there are papers on asymptotic behaviors (see e.g.
  \cite{KageiMaekawa:20122951}, \cite{Winkler:20102889}) and stationary solutions (see e.g.
 \cite{GuoHwang:20101519}, \cite{LinNiTakagi:19981}).
  We also refer readers to, for instance
  \cite{IshidaYokota:20122469} and references cited therein, to see
  results on the quasilinear degenerate Keller-Segel system.

 The {first goal} of this paper is
  to answer {{Kozono,\;Sugiyama and Wachi}}'s question in \cite{KozonoSugiyamaWachi:20121213} of figuring out whether there exists a
  solution to system \eqref{k-s-1}--\eqref{k-s-3} even locally in time for $(u_0, v_0)\in{L}^1(\mathbb{R}^2)\times{L}^\infty(\mathbb{R}^2)$.
 In fact, we prove that there does exist a unique global {\it mild solution} to
  system \eqref{k-s-1}--\eqref{k-s-3} with $(u_0,v_0)\in{L^1(\mathbb{R}^2)\times{L}^\infty}(\mathbb{R}^2)$ by estimating $t\|u(t,\cdot)\|_{L^\infty}$, $\|u(t,\cdot)\|_{L^1}$ and ${t}^{\frac{1}{2}}\|\nabla{v}(t,\cdot)\|_{L^\infty}$
  in the $L^p$-framework, see for instance \cite{Kato:1984471}.
  Moreover, by exploring the special structure
  of system \eqref{k-s-1}--\eqref{k-s-2},  Deng and Li \cite{Deng:2012} established global existence of mild
  solution for initial data $(u_0,v_0)\in{L}^q(\mathbb{R}^2)\times{L}^\infty(\mathbb{R}^2)$ with
  $1<q<\infty$,  where global existence of mild solution for initial data $(u_0,v_0)\in{L^\infty(\mathbb{R}^2)}\times{L}^\infty(\mathbb{R}^2)$
  was left as {\it an open} question.

 The {second goal} of this paper is to study global well-posedness of system \eqref{k-s-1}--\eqref{k-s-3}
 with $H^1_b(\mathbb{R}^2)\times H^1(\mathbb{R}^2)$
 initial data. Up to now, there are several results on local and global existence
 of system \eqref{k-s-1}--\eqref{k-s-3} for $(u_0,v_0)\in{H}^{\nu}(\mathbb{R}^2)\times{H}^{\nu}(\mathbb{R}^2)$
 with $\nu >\!1$ (cf. Nagai, Senba and Yoshida
 \cite{NagaiSenbaYoshida:1997411}, and Yagi \cite{Yagi:1997241}) and result on global existence of system \eqref{k-s-1}--\eqref{k-s-3} with initial data
 $(u_0,v_0)\in H^{-1}(\mathbb{R}^2)\times{H}^1(\mathbb{R}^2)$ (cf. \cite{Deng:2012}).
 Recalling that $H^{\nu}(\mathbb{R}^2)$ $\!\hookrightarrow\!{H}^1_b(\mathbb{R}^2)$
 and $H^1(\mathbb{R}^2)$ can not be embedded into
 ${L}^\infty(\mathbb{R}^2)$, hence
  global existence of mild solution to system \eqref{k-s-1}--\eqref{k-s-3}
 with $(u_0,v_0)\in{H}^1_b(\mathbb{R}^2)\times{H}^1(\mathbb{R}^2)$ improves the previous
 results. The proof
 is based on a combination of the $L^2$-Fourier multiplier theory, the smoothing properties of heat
 kernel and the {\it new half norm} of $v$, i.e. $\sup_{t>0}{t^{\frac{1}{2}}}{(1+t)^{-\frac{1}{2}}}\|\nabla v\|_{L^\infty}$
 which balances the need for $t$ near zero and $t$ near infinity.
 With this unusual half norm, different form the usual scaling
 invariant ones, enables us to overcome the main difficulty and to
 close the iteration scheme.
 At last, global well-posedness of system \eqref{k-s-1}--\eqref{k-s-3} with initial data
  $(u_0,v_0)\in H^1(\mathbb{R}^2){\times}H^1(\mathbb{R}^2)$ is left
  as another {\it open} question.

 \!Next we recall some results concerning the parabolic-elliptic/parabolic-hyperbolic Keller-Segel systems.
 Concerning the parabolic-elliptic Keller-Segel model
\begin{align*}
 {{u}_t} = \Delta{u} - \nabla\cdot({u}\nabla{v}),\quad\quad
       & \Delta{v} = {v} - {u}.
  \end{align*}
It was conjectured by Childress and Percus
\cite{ChildressPercus:198461} that in a two-dimensional domain
$\Omega$ there exists a critical number $c^\ast$ such that if
$\int_\Omega u_0(x)dx <c^\ast$ then the solution exists globally in
time, and if $\int_\Omega u_0(x)dx > c^\ast$ then blowup happens.
For different versions of the Keller-Segel model, the conjecture has
been essentially proved; for a complete review of this topic, we
refer the reader to the paper \cite{Horstmann:2003103} and the
references therein, also see e.g. Diaz, Nagai, and Rakotoson
\cite{DiazNagaiRakotoson:1998156}, Blanchet, Dolbeault and Perthame
\cite{BlanchetDolbeaultPerthame:20061}. As for the
hyperbolic-hyperbolic Keller-Segel model
 \begin{align*}
 \partial_tu=\Delta u+\nabla\cdot(u\nabla{w}),\quad\quad \partial_t
 {w}=u,
  \end{align*}
 it was used in \cite{WangHillen:200845}
 for one dimensional case and was extended to
 multidimensional cases in \cite{LiPanZhao:2011}, and has been studied in
 \cite{LevineSleeman:1997683,OthmerStevens:19971044}
 and a comprehensive qualitative and numerical analysis was provided there. We refer readers to references
  \cite{CorriasPerthameZaag:2003141,CorriasPerthameZaag:20041,Eisenbach:2004,KellerSegel:1971225,KellerSegel:1971235,
 L-W2,N-I,Patlak:1953311,SleemanWardWei:2005790,YangChenLiuSleeman:2005432} for more
 discussions in this direction.

  Throughout this paper, both $\mathcal{F}f$ and $\widehat{f}$ stand for Fourier
 transform of $f$ with respect to space variable and
 $\mathcal{F}^{-\!1}$ stands the inverse Fourier transform. Let $C$ and $c$ be positive constants that may vary from line to line. $A\!\lesssim\! B$ stands
 for $A\!\le\!{CB}$ and $A\sim{B}$
 stands for $A\lesssim{B}\lesssim{A}$.
 For any $(p,q)\in[1,\infty]^2$, we denote
 $L^p(0,\infty)$, $L^q(\mathbb{R}^2)$, $H^s(\mathbb{R}^2)$ and $L^p(0,\infty;L^q(\mathbb{R}^2))$ by $L^p_t$, $L^q$, $H^s$ and $L^p_tL^q$, respectively.

\medskip
 \begin{theorem}\label{thm:1.1}
  For any initial data $(u_0,v_0)\!\in\!{L}^1(\mathbb{R}^2)\times L^\infty(\mathbb{R}^2)$
  with ${\sup_{t>0}}\|e^{t\Delta}u_0\|_{L^1}$ and ${\sup_{t>0}}t^{\frac{1}{2}}\|\nabla{e}^{t\Delta}v_0\|_{L^\infty}$
 being small, there exist a unique global mild solution $(u,v)$
  to system  \eqref{k-s-1}--\eqref{k-s-3} and positive constant $c$ such that
 \begin{align*}
   (u,v)\in {C}([0,\infty);L^1(\mathbb{R}^2))\times{C}_w([0,\infty);L^\infty(\mathbb{R}^2))\end{align*}
   with {$C_w([0,\infty);X)$ being the set of weakly-star continuous functions on $[0,\infty)$ valued in Banach space
   $X$}, and
 \begin{align}\label{eq:1.4}
 \displaystyle{\sup_{t>0}}\,(\|u\|_{L^1}&+t\|u\|_{L^\infty}+\frac{1}{4c}t^{\frac{1}{2}}\|{\nabla{v}}\|_{L^\infty})\nonumber\\
 &\le{2}\,\displaystyle{\sup_{t>0}}\,(\|e^{t\Delta}u_0\|_{L^1}+t\|e^{t\Delta}u_0\|_{L^\infty}+
  \frac{1}{4c}t^{\frac{1}{2}}\|e^{t\Delta}{\nabla{v}_0}\|_{L^\infty}),\end{align}
 which yields that $\|u\|_{L^\infty}\le{o}(t^{-1})$ and
  $\|\nabla v\|_{L^\infty}\le o(t^{-\frac{1}{2}})$ as $t\rightarrow \infty$.
 \end{theorem}
 \noindent \textbf{Remark:}\! ($i$) Applying Lemma
\ref{lem:2.5} to Proposition \ref{pro:3.1}, we observe that
\eqref{eq:1.4} holds if
 \begin{align}\label{eq:1.5}
    {\sup_{t>0}}\,(\|e^{t\Delta}u_0\|_{L^1}+t\|e^{t\Delta}u_0\|_{L^\infty}+
      \frac{1}{4c}t^{\frac{1}{2}}\|e^{t\Delta}{\nabla{v}_0}\|_{L^\infty})\le\frac{3}{32c^2}.
  \end{align}
 Applying $\|\frac{1}{4\pi{t}}e^{-\frac{|\cdot|^2}{4t}}\|_{L^p}\le t^{-1+\frac{1}{p}}$ and Proposition
 \ref{def:2.4} to the left hand side of \eqref{eq:1.5}, it suffices to assume that
 $2\|u_0\|_{L^1}+\frac{1}{4c}\|\nabla{v}_0\|_{\dot{B}^{-1}_{\infty,\infty}}\!\le\!\frac{3}{32c^2}$,
 where
 \begin{align*}\|v_0\|_{\dot{B}^0_{\infty,\infty}}\sim\|\nabla{v}_0\|_{\dot{B}^{-1}_{\infty,\infty}}=\sup_{t>0}t^{\frac{1}{2}}\|e^{t\Delta}\nabla
 v_0\|_{L^\infty}\le
 \|v_0\|_{L^\infty}\end{align*}
 since Riesz transforms $\frac{\nabla}{\sqrt{-\Delta}}$ are bounded in homogeneous Besov
 spaces,  $\sqrt{-\Delta}$ maps $\dot{B}^{0}_{\infty,\infty}$ isomorphically
   onto $\dot{B}^{-1}_{\infty,\,\infty}$ and
   $\frac{1}{\sqrt{-\Delta}}$ maps $\dot{B}^{-1}_{\infty,\infty}$ isomorphically
   onto $\dot{B}^{0}_{\infty,\,\infty}$ (see \cite{Triebel:1983}, {\small Theorem 1,
   p.242}). Therefore, only $\dot{B}^0_{\!\infty,\infty}$ smallness
 of $v_0$ and $L^1$ smallness\clearpage \noindent of $u_0$ are needed.
 Local existence of mild solution follows directly by  changing time interval
 $[0,\infty)$ into $[0,T]$. However, if  $v_0(x_1,x_2)=1_{[0,1]}(x_1)$ and
  $(t,x_1,x_2)\in(0,\frac{1}{64})\times({{t}^{\frac{1}{2}}}\!,\,2{{t}^{\frac{1}{2}}})\times(-\infty,\infty)$, then there holds
 \begin{align*}
  t^{\frac{1}{2}}|\partial_1 e^{t\Delta}v_0(x_1,x_2)|&=\!\!\int_{-\infty}^\infty\!\int_{-\infty}^\infty
   \frac{1}{4\pi{t}}\frac{|x_1\!-\!y_1|}{\sqrt{4t}}e^{-\frac{(x_1-y_2)^2+(x_2-y_2)^2}{4t}}1_{[0,1]}(y_1)dy_1dy_2\\
  &=\!\frac{1}{\sqrt{\pi}}\!\int_{y_1\in[0,1]}\!\frac{|x_1\!-y_1|}{\sqrt{4t}}e^{-\frac{(x_1-y_1)^2}{4t}}d\frac{y_1-x_1}{\sqrt{4{t}}}\\
  &=\!\frac{1}{\sqrt{\pi}}\!\int_{\frac{-x_1}{\sqrt{4t}}}^{\frac{1-x_1}{\sqrt{4t}}}
     \!ze^{-z^2}dz=\!\frac{1}{\sqrt{\pi}}\!\int_{\frac{x_1}{\sqrt{4t}}}^{\frac{1-x_1}{\sqrt{4t}}}
     \!ze^{-z^2}dz\\
  &\ge \!\frac{1}{\sqrt{\pi}}\!\int_{1}^{3}
     \!ze^{-z^2}dz=c_0.
 \end{align*}
 For such $v_0$, if we set $\widetilde{v}_0=v_0/c_0$, then we have
 $${\lim_{T\rightarrow0^+}\sup_{0<t<T}}t^{\frac{1}{2}}\|\nabla
 e^{t\Delta}\widetilde{v}_0\|_{L^\infty}\ge 1.$$ Hence it seems difficult to prove local (global) existence of mild solution for arbitrary large
 $L^1(\mathbb{R}^2)\times{L}^\infty(\mathbb{R}^2)$
 initial data.

 \noindent($ii$)\, Proof of Theorem \ref{thm:1.1} also applies for
 $u_t-\Delta u+\nabla\cdot({u\nabla{v}})=0$ and $v_t-\Delta v+u=0$ with initial data $(u_0,v_0)\in
 L^1(\mathbb{R}^2)\times{L^\infty}(\mathbb{R}^2)$.

\medskip\medskip
 Here and hereafter, we set
 $\sigma(t)\!=t^\frac{1}{2}(1+t)^{-\frac{1}{2}}$. Then we state the
 following result.
 \begin{theorem}\label{thm:1.2}
 For any initial data
 $(u_0, v_0)\in{H^1_b}(\mathbb{R}^2)\times H^1(\mathbb{R}^2)$, there exist positive constants $\varepsilon_0$ and $c$ so that
 if $\|u_0\|_{L^\infty}+\|u_0\|_{H^1}+\frac{1}{4c}\|v_0\|_{H^1}\le \varepsilon_0$, then
  system \eqref{k-s-1}--\eqref{k-s-3} has a unique global solution $(u,v)$ satisfying
 \begin{align*}
    & (u\pm{v},\,u\pm{\sigma}\hskip.01cm \nabla{v},\,\nabla{u}\pm\nabla v) \in {C}([0,\infty);H^1(\mathbb{R}^2))\times{L}^{\!\infty}_t\!L^{\!\infty}
    \times {L}^2_tH^1.
 \end{align*}
  Moreover,
  $\|u\pm\frac{v}{4c}\|_{L^{\infty}_tH^1}+\|u\pm\frac{{\sigma}}{4c}\nabla{v}\|_{L^{\infty}_tL^{\infty}}+\|\nabla{u}\|_{L^{2}_tL^2}+\frac{1}{4c}
  \|\nabla{v}\|_{L^{2}_tH^1}\le{2}\hskip.02cm\varepsilon_0$.
 \end{theorem}

\medskip

\noindent \textit{Plan of the paper:} In Sect. \!2 we introduce
 several preliminary lemmas, while in Sect. \!3 we prove Theorems \ref{thm:1.1} and \ref{thm:1.2}.
 \section{Preliminaries}
In this section, we list several known lemma and prove some key
lemmas which will be used in proving the well-posedness of the
parabolic-parabolic chemotaxis. The first lemma given below is
concerned with initial data belonging to $H^s(\mathbb{R}^2)$. For
simplicity, here and hereafter, we omit the space domain in various
function spaces, for instance $H^1(\mathbb{R}^2)$ is denoted by
$H^1$, if there is no confusion.
 \begin{lemma}\label{lem:2.1}
 Let $n=2$, $\Lambda=\sqrt{-\Delta}$, $(s,\hskip.04cm \delta,\hskip.04cm  r,\hskip.04cm \rho)\in(-\infty,\infty)\times[0,\infty)\times[1,\infty]\times[2,\infty]$ and $v\in{H}^s$.
 If $m(t,\xi)\in{L}^{r}_tL^{\!\infty}_\xi$ and
 $m(t,D)v=\!\mathcal{F}^{-1}m(t,\xi)\widehat{v}(\xi)$, then we get
 \begin{align}\label{eq:2.1}
      \|m(t,D)v\|_{L^r_tH^s}\lesssim\|m\|_{{L}^{r}_tL^{\infty}_\xi}\|v\|_{H^s};
  \end{align}
 Else if ${m}_\delta(t,\xi):=m(t,\xi)|\xi|^{\delta}\in\!{L^{\!\infty}_\xi{L}^{\rho}_t}$ and
 $m_\delta(t,D)v\!=\!\mathcal{F}^{-1}m_\delta(t,\xi)\widehat{v}(\xi)$, then we get
 \begin{align}\label{eq:2.2}
      \|m_\delta(t,D)v\|_{L^{\rho}_t{H}^{s}}\lesssim\|m_\delta\|_{L^\infty_\xi{L}^{\rho}_t}\|v\|_{H^s}.
  \end{align}
 \end{lemma}
 \begin{proof}
  Proof of \eqref{eq:2.1} follows from classical Fourier multiplier theory
  and, for readers convinience, we give the detail proof as follows:
 \begin{align}\label{eq:2.3}
      \|m(t,D)v\|_{L^r_tH^s}
      &\lesssim\|m(t,\xi)(1+|\cdot\!|^2)^{\frac{s}{2}}\widehat{v}(\cdot)\|_{L^r_tL^2_\xi}\nonumber\\
      &\lesssim\|m\|_{{L}^{r}_tL^{\infty}_\xi}\|(1+|\cdot\!|^2)^{\frac{s}{2}}\widehat{v}(\cdot)\|_{L^2_\xi}\;\;\nonumber\\
      &\lesssim\|m\|_{{L}^{r}_tL^{\infty}_\xi}\|v\|_{H^s},
  \end{align}
 where we have used Plancherel equality twice.

   In order to prove \eqref{eq:2.2}, by making use of Plancherel equality, Minkowski's inequality, H\"{o}lder's inequality and Plancherel equality again,
   we get
 \begin{align}
   \|m_\delta(t,D){v}\|_{L^{\rho}_t{H}^{s}}
  & \lesssim\|m_\delta\;(1+|\cdot\!|^2)^{\frac{s}{2}}\widehat{v}(\cdot)\|_{L^{\rho}_tL^2_\xi}\nonumber\\
  & \lesssim\|m_\delta\;(1+|\cdot\!|^2)^{\frac{s}{2}}\widehat{v}(\cdot)\|_{L^2_\xi{L^{\rho}_t}}\nonumber\\
  & \lesssim \|m_\delta\|_{L^\infty_\xi{L}^{\rho}_t}\|(1+|\cdot\!|^2)^{\frac{s}{2}}\widehat{v}(\cdot)\|_{L^2_\xi}\nonumber\\
  & \lesssim
  \|m_\delta\|_{L^\infty_\xi{L}^{\rho}_t}\|v\|_{H^s}.\label{eq:2-4}
 \end{align}
 Hence, we finish the proof.
 \end{proof}
  The skill used in the above Lemma will be used repeatedly in the following parts.
  The next Lemma is devoted to estimate the bilinear term which is known as
  the maximal $L^p_tL^q$ regularity result for the heat kernel (cf. \cite{LemarierRieusset:2002}, Theorem 7.3, p.\,64).
 \begin{lemma}\label{lem:2-2} $({\rm Maximal}$ $L^{p}_tL^q$ ${\rm regularity\; for\; heat\; kernel} )$\
 The operator $T$ defined by
 \begin{align}\label{eq:2.5}
     g(t,x)\mapsto{Tg(t,x)}=\!\int_0^t\!e^{(t-\tau)\Delta}\Delta{g}(\tau,x)d\tau
 \end{align}
 is bounded from $L^p_tL^q$ to $L^p_tL^q$ with $1\!<p<\!\infty$ and $1\!<q\!<\infty$.%
 \end{lemma}
 The next Lemma is also dedicated to estimating the bilinear term.
 \begin{lemma}\label{lem:2.3}
 For any $(s,\hskip.02cm c,\hskip.02cm c_1,\hskip.02cm p,\hskip.02cm  r, \hskip.02cm
 p_1)\in\mathbb{R}\times(0,\infty)^2\times[2,\infty]\times[1,2]\times[1,\infty]$, ${p}_1\ge{r}$ and
 $0\le \theta \!< 2(1+\!\frac{1}{p_1}-\!\frac{1}{r})$, if\hskip.2cm $m(t,\xi) = \frac{c_1}{e^{\,ct|\xi|^2}}$
 and $\mu(t,\xi)=\frac{{c}_1}{e^{\,ct+ct|\xi|^2}}$, then there exists constant $C_{\theta,p_1,r}$ depending on $\theta$, $p_1$ and
 $r$ such that
 have
 \begin{align}\label{eq:2.6}
  &\|\!\int_0^t\!\!m(t\!-\!\tau,D)\Lambda^{{2+\frac{2}{p}-\frac{2}{r}}} F(\tau,x)\,d\tau\|_{L^p_tH^s}\lesssim\|F\|_{L^r_tH^s},
 \\ \label{eq:2.7}
  &\|\!\int_0^t\!\mu(t\!-\!\tau,D)\Lambda^{\theta}F(\tau,x)\,d\tau\|_{L^{p_1}_t\dot{H}^s}\lesssim{C}_{\theta,p_1,r}\|F\|_{L^r_t\dot{H}^s}.
  \end{align}
 \end{lemma}
 \begin{proof}
 In order to prove \eqref{eq:2.6}, setting $\langle\xi\rangle=\sqrt{1+|\xi|^2}$ and by using Plancherel equality, the
 Minkowski inequality, the Young inequality and the Minkowski inequality  as well as Plancherel equality, we have
  \begin{align*}
   \|\!\int_0^t\!m(t\!-\!\tau,D)&\Lambda^{2+\frac{2}{p}-\frac{2}{r}}F(\tau,x)d\tau\|_{L^{p}_t{H}^{s}}
   \!=\!\|\!\int_0^t\!m(t\!-\!\tau,\xi)|\xi|^{2+\frac{2}{p}-\frac{2}{r}}\langle\xi\rangle^s\widehat{F}(\tau,\xi)d\tau\|_{L^p_tL^2_\xi}&{}\\
   &\lesssim\|\int_0^tm(t-\tau,\xi)|\xi|^{2+\frac{2}{p}-\frac{2}{r}}\langle\xi\rangle^s\widehat{F}(\tau,\xi)d\tau\|_{L^2_\xi{L}^p_t} &{}\\
   &\lesssim\Big\|\|m(\cdot,\xi)|\xi|^{2+\frac{2}{p}-\frac{2}{r}}\|_{L^{\frac{p\hskip.04cm r}{p\hskip.04cm r+r-p}}_t}\|\langle\xi\rangle^s\widehat{F}(\cdot,\xi)\|_{L^r_t}\Big\|_{L^2_\xi} &{}\\
   &\lesssim\sup_{\xi\in\mathbb{R}^2}\|m(\cdot,\xi)|\xi|^{2+\frac{2}{p}-\frac{2}{r}}\|_{L^{\frac{p\hskip.04cm r}{p\hskip.04cm r+r-p}}_t}\|\langle\xi\rangle^s\widehat{F}(\cdot,\xi)\|_{L^2_\xi L^r_t} &{}\\
   &\lesssim\|\langle\xi\rangle^s\widehat{F}(\cdot,\xi)\|_{ L^r_tL^2_\xi}&{}\\
   &\lesssim\|F\|_{L^r_tH^s}.
  \end{align*}
 It remains to prove \eqref{eq:2.7}. Using the $L^1$ integrability of $e^{-ct}t^{-\frac{\theta}{2+\frac{2}{p_1}-\frac{2}{r}}}$, we
 get
  \begin{align*}
     \|\!\int_0^t\!\mu(t\!-\!\tau,D)\Lambda^{\theta}&F(\tau,x)d\tau\|_{L^{p_1}_t\dot{H}^{s}}
   = \|\int_0^te^{-c(t-\tau)}(t\!-\!\tau)^{-\frac{\theta}{2}}\|{F}\|_{\dot{H}^s}d\tau\|_{L^{p_1}_t}&{}\\
   &\lesssim\Big(\int_0^t e^{-c{\frac{p_1\hskip.04cm r}{p_1\hskip.04cm r+r-p_1}}t}
             t^{-\frac{\theta {p_1\hskip.04cm r}} {2(p_1\hskip.04cm r+r-p)} }dt\Big)^{1+\frac{1}{p_1}-\frac{1}{r}}\|F\|_{L^r_t\dot{H}^s} &{}\\
   &\lesssim{C}_{\theta,p_1,r}\|F\|_{L^r_t\dot{H}^s}.
  \end{align*}
  Therefore, we finish the whole proof.
 \end{proof}
 Let us state the equivalent definition of Besov space $\dot{B}^s_{p,q}:=\dot{B}^s_{p,q}(\mathbb{R}^2)$ using heat
 semigroup method (for a proof see, for instance \cite{Triebel:1983} {\small p.192} or \cite{LemarierRieusset:2002}  {\small Theorem 5.4, p.45}).
 \begin{proposition}\label{def:2.4}\!
 Let $(s,p,q)\!\in\!(-\infty,0)\times\![1,\infty]^2\!.\!$ The homogeneous Beosv space
 $\dot{B}^s_{\!p,q}$ is defined as the set of tempered distribution $f$ such that
  \begin{align*}
  &\|f\|_{\dot{B}^{s}_{p,q}}=\Big(\int_0^\infty(t^{\frac{-s}{2}}\|e^{t\Delta}f\|_{L^p})^q\frac{dt}{t}\Big)^{\frac{1}{q}}\quad
   \;\text{ if }\; 1\le q<\infty,&{}\\
  &\|f\|_{\dot{B}^s_{p,\infty}}\!=\sup_{t>0}t^{\frac{-s}{2}}\|e^{t\Delta}f\|_{L^p}\quad\hskip1.62cm\text{ if  }\; q=\infty.\hskip.7cm
  \end{align*}
 \end{proposition}

 The last lemma of this section is a slightly generalized version about the well-known
 Picard contraction principle (see for instance \cite{LemarierRieusset:2002}, {\small Theorem 13.2, p.124}) which is used to prove
 the main results concerning well-posedness of \eqref{k-s-1}--\eqref{k-s-3} with
 $(u_{0},v_{0})$ either belonging to $L^1(\mathbb{R}^2)\times L^\infty(\mathbb{R}^2)$ or belonging to $H^1_b(\mathbb{R}^2)\times H^1(\mathbb{R}^2)$.
\begin{lemma}\label{lem:2.5} $(\mathrm{The\; Picard\; contraction\; principle})$\;
  Let
  $({X}\times{Y},\; \|\cdot\!\|_{X} \!+\! \|\cdot\!\|_{Y})$  be an abstract Banach product space,
  $L: {X}\rightarrow{Y}$ and
  $B:{X}\times{Y}\rightarrow{X}$ are a linear operator and a bilinear operator, respectively, such that for any 
  $(u,v)\in{X}\times{Y}$, there exist positive constant $c$ and if
 \begin{align}\label{eq:2.8}
  \|L(u)\|_{Y}\le c\|u\|_{X},\;\;\|B(u,v)\|_{X}\le{c}\|u\|_{X}\|v\|_{Y},
 \end{align}
  then for any $(e^{t\Delta}u_{0},e^{t(\Delta-1)}v_{0})\in{X\times{Y}}$ with
  $\|(e^{t\Delta}u_{0},\frac{1}{4c}e^{t(\Delta-1)}v_0)\|_{X\times{Y}}<\frac{3}{32c^2}$, the following system
 \begin{equation}\label{eq:2.9}
     (u, v)=(e^{t\Delta}u_{0},e^{t(\Delta-1)}v_{0}) + (B(u,v),\ L(u))
 \end{equation}
 has a solution $(u,v)$ in $ {X}\times {Y}$. In particular, the solution is such that
  $\|(u,\frac{v}{4c})\|_{ {X}\times {Y}}\le{2}\|(e^{t\Delta}u_{0},\frac{1}{4c}e^{t(\Delta-1)}v_0)\|_{ {X\times
  {Y}}}$
 and it is the only one such that $\|(u,\frac{v}{4c})\|_{ {X}\times {Y}}<\frac{3}{16c^2}.$
\end{lemma}
\begin{proof}
 The proof is standard now. However, for reader's convenience, we give a brief proof.
 We first define a mapping ${\Phi}: X\times{Y}\rightarrow X\times
 Y$ such that
 \begin{align}
 \label{eq:2.11}
 \Phi(u,v)=(e^{t\Delta}u_{0},e^{t(\Delta-1)}v_{0}) + (B(u,v),\ L(u)).
 \end{align}
 Applying simple transformations, i.e. $w=\frac{v}{4c}$ and $w_0=\frac{1}{4c}v_0$ to \eqref{eq:2.11}, we get
 \begin{align}\label{eq:2.12}
 \Phi(u,w)=(e^{t\Delta}u_{0},e^{t(\Delta-1)}w_0) + (4cB(u,w),\frac{1}{4c}L(u)).\end{align}
 By applying \eqref{eq:2.8} to \eqref{eq:2.12}, we have
 \begin{align}\label{eq:2.13}
  \|\Phi(u,w)\|_{X\times{Y}}
  &\le\|(e^{t\Delta}u_{0},e^{t(\Delta-1)}w_0)\|_{X\times Y}+4c^2\|u\|_X\|w\|_Y+\frac{1}{4}\|u\|_Y\nonumber\\
  &\le {A_0}+ c^2\|(u,w)\|_{X\times{Y}}^2+\frac{1}{4}\|(u,w)\|_{X\times{Y}},
 \end{align}
 where $A_0:=\|(e^{t\Delta}u_{0},e^{t(\Delta-1)}w_0)\|_{X\times Y}$.
  Let $\overline{B(0,2A_0)}\subset{X\times Y}$ be a closed ball centered at origin
 with radius $2A_0$. From \eqref{eq:2.13}, we observe that
 $\Phi$ is well defined in $\overline{B(0,2A_0)}$ and maps
 $\overline{B(0,2A_0)}$ into itself. Moreover, for any $(u_1,w_1)$,
 $(u_2,w_2)\in\overline{B(0,2A_0)}$, by making use of \eqref{eq:2.8}, we
 get
 \begin{align}\label{eq:2.14}
 &\|\Phi(u_1,w_1\!)-\!\Phi(u_2,w_2)\|_{X\times Y}\!=\!\|(4cB(u_1,w_1\!)\!-\!4cB(u_2,w_2),\, \frac{L(u_1\!)\!-\!L(u_2)}{4c})
 \|_{X\times{Y}}\nonumber\\
 &\le4c^2\max\{\|u_2\|_X,\|w_1\|_Y\}\|(u_1\!-u_2,w_1\!-w_2)\|_{X\times{Y}}+\frac{1}{4}\|(u_1\!-u_2,w_1\!-w_2)\|_{X\times{Y}}\nonumber\\
 &\le8c^2A_0\|(u_1-u_2,w_1-w_2)\|_{X\times{Y}}\!+\!\frac{1}{4}\|(u_1-u_2,w_1-w_2)\|_{X\times{Y}}\nonumber\\
 &\le (8c^2A_0+\frac{1}{4})\|(u_1-u_2,w_1-w_2)\|_{X\times{Y}},
 \end{align}
 where $8c^2A_0+\frac{1}{4}<1$ since $A_0<\frac{3}{32c^2}$.
 From \eqref{eq:2.14}, we observe that $\Phi: (u,w)\mapsto \Phi(u,w)$
 in \eqref{eq:2.12} is contractive. Thus there exists a unique
 solution $(u,w)$ to \eqref{eq:2.12}, which shows that
 \eqref{eq:2.11} also has a unique solution $(u,v)$ to \eqref{eq:2.11}
 provides that $\|(e^{t\Delta}u_{0},\frac{1}{4c}e^{t(\Delta-1)}v_0)\|_{X\times
 Y}<{3}/{32c^2}$.
\end{proof}
\section{Proof of Theorem \ref{thm:1.1}}
 As usual, we apply the heat semigroup $e^{t\Delta}$ with heat kernel $\frac{1}{4\pi t}e^{-\frac{|x|^2}{4t}}$ to invert
 system \eqref{k-s-1}--\eqref{k-s-3} into the following integral equations via the Duhamel principle:
\begin{align}\label{eq3-1}
  \hskip-1.051cm \left\{\begin{aligned}{u}&=e^{t\Delta}{u}_0-\int_{0}^te^{(t-\tau)\Delta}\nabla\cdot({u}\nabla{v})d\tau:=e^{t\Delta}{u}_0-B(u,v),\\
     v&=e^{t(\Delta-1)}v_0+\int_{0}^te^{(t-\tau)(\Delta-1)}{u}d\tau:=e^{t(\Delta-1)}v_0+L(u).
    \end{aligned}\right.
   \end{align}
 Let $c$ be the largest positive constant that appears in the
 linear and bilinear estimates and depends only on dimension.
 By denote $\frac{v}{4c}$ by $w$, we get the following system
\begin{align}\label{eq3-2'}
   \left\{\begin{aligned}
   {u}&=e^{t\Delta}{u}_0-4c\int_{0}^te^{(t-\tau)\Delta}\nabla\cdot({u}\nabla{w})d\tau:=e^{t\Delta}{u}_0-4cB(u,w),\\
    w &=e^{t(\Delta-1)}w_0+\frac{1}{4c}\int_{0}^te^{(t-\tau)(\Delta-1)}{u}d\tau:=e^{t(\Delta-1)}w_0+\frac{1}{4c}L(u),
   \end{aligned}\right.
   \end{align}
 where we regard equations \eqref{eq3-2'} as a fixed point system and let mapping
 $\Phi$ be
 \begin{align}\Phi:
    (u,w)\mapsto
    \Big(e^{t\Delta}u_0,\,e^{t(\Delta-1)}{w_0}\Big)+\Big(-4cB(u,w),\,\frac{1}{4c}L(u)\Big).\label{eq3-3}\quad\quad
 \end{align}
 We call solution $(u,4cw)$ to \eqref{eq3-1} {\it
 mild solution} of \eqref{k-s-1}--\eqref{k-s-3} if $(u,w)$ solves \eqref{eq3-2'}.

\subsection{Proof of Theorem \ref{thm:1.1}}
 In this subsection, we prove global well-posedness of system \eqref{eq3-2'} with initial data
  $(u_0,w_0)\in L^1(\mathbb{R}^2)\times L^\infty(\mathbb{R}^2)$ by making
 use of the Kato's $L^p$-framework. At first, we
 set
 \begin{align}\label{eq:3-4}
   \begin{aligned}
    &X = \{u\in \mathcal{S}'(\mathbb{R}^2\times (0,\infty)); \ \sup_{t>0} \|u(\cdot,t)\|_{L^1} + \sup_{t>0}\; t\|u(\cdot,t)\|_{L^\infty}
           < \infty\;\},\\
    &Y= \{w\in \mathcal{S}'(\mathbb{R}^2\times (0,\infty)); \ 
    \sup_{t>0} t^{\frac{1}{2}}\|\nabla{w}(\cdot,t)\|_{L^\infty}
          \!<\! \infty\}.
   \end{aligned}
  \end{align}
 Then we prove that for suitably small initial data $(u_0,w_0)$ the
 mapping $\Phi$ is contractive and maps a closed ball of $X \times {Y}$ into itself.
\begin{proposition}\label{pro:3.1}
  For any initial data $({u}_0,w_0)\in{L}^1\times{L}^\infty$, there exists positive constant $c$ such that
  \begin{align}\label{eq:3-5}
  \left\{\begin{aligned}&\|(e^{t\Delta}u_0,e^{t(\Delta-1)}w_0)\|_{X\times{Y}}\le
   c\hskip.02cm\|(u_0,w_0)\|_{L^1\times \dot{B}^0_{\infty,\infty}}\!\!\le c\hskip.02cm\|(u_0,w_0)\|_{L^1\times L^\infty},\\
  &\|4cB(u,w)\|_X\le 4c^2\|u\|_X\|w\|_Y\;\;\;\text{ and }\;\;\;
  \|\frac{1}{4c}L(u))\|_{{Y}}\le\frac{1}{4}\|u\|_X.\end{aligned}\right.
 \end{align}
 \end{proposition}
 \begin{proof}
   We divide the whole proof into two parts concerning with $e^{t\Delta}u_0$,
   $e^{-t}e^{t\Delta}w_0$ and
   $B(u,w),$ $L(u)$, respectively.

\noindent   \textsc{Part I.}\; {\it Estimates for
 $\|e^{t\Delta}u_0\|_X$ and $\|e^{-t}e^{t\Delta}w_0\|_Y$.}
 Recall that the heat kernel is
 $\frac{1}{4\pi{t}}e^{-\frac{|x|^2}{4t}}$. Then for any $t>0$ and $1\le
 p\le\infty$, there hold
   \begin{align}\label{eq:3-6}
   \|\frac{1}{4\pi{t}}e^{-\frac{\,\,\,|\cdot|^2}{4t}}\|_{L^p}\le {t}^{-1+\frac{1}{p}}\quad\text{and }\
   \|\frac{1}{4\pi{t}}\nabla e^{-\frac{\,\,|\cdot|^2}{4t}}\|_{L^p}\le t^{-\frac{3}{2}+\frac{1}{p}}.
   \end{align}
   Applying Young's inequality and \eqref{eq:3-6} to
   $e^{t\Delta}u_0(x)\!=\!\int_{\!\mathbb{R}^2}\!\frac{1}{4\pi
   t}e^{-\frac{|y|^2}{4t}}u_0(x\!-\!y)dy$,
   we get
   \begin{align}\label{eq:3-7}
    \|e^{t\Delta}u_0\|_X&=\sup_{t>0}\|e^{t\Delta}u_0\|_{L^1}+\sup_{t>0}\,t\hskip.01cm\|e^{t\Delta}u_0\|_{L^{\!\infty}}\nonumber\\
                        &\le \sup_{t>0}\|\frac{1}{4\pi t}e^{-\frac{|\cdot|^2}{4t}}\|_{L^1}\|u_0\|_{L^1}+
                        \sup_{t>0}t\|\frac{1}{4\pi t}e^{-\frac{|\cdot|^2}{4t}}\|_{L^\infty} \|u_0\|_{L^1}\nonumber\\
                        &\le2\|u_0\|_{L^1}
   \end{align}
   or from Proposition \ref{def:2.4} and embedding theorem $L^1\hookrightarrow\dot{B}^{-2}_{\infty,\infty}$, we have
   $$\sup_{t>0}\,t\hskip.01cm\|e^{t\Delta}u_0\|_{L^{\!\infty}}=\|u_0\|_{\dot{B}^{-2}_{\infty,\infty}}\le{c_1}\|u_0\|_{L^1}.$$
   We emphasize here that for any $(s,\alpha,p,q)\in\mathbb{R}^2\times[1,\infty]^2$, $(-\Delta)^{\frac{\alpha}{2}}$ maps $\dot{B}^{s}_{p,q}$ isomorphically
   onto $\dot{B}^{s-\alpha}_{p,\,q}$ (cf. \cite{Triebel:1983}, Theorem 1,
   p.242), which is a direct consequence of the well known
   Bernstein's inequalities. Thus following similar arguments of $e^{t\Delta}$ and using \eqref{eq:3-6} as well as Proposition \ref{def:2.4}, we get
   \begin{align}
   \|e^{-t}e^{t\Delta}w_0\|_Y&=\sup_{t>0}\,t^{\frac{1}{2}}\hskip.01cm
   \|e^{-t}e^{t\Delta}\nabla w_0\|_{L^{\!\infty}}\le\sup_{t>0}\,t^{\frac{1}{2}}\hskip.01cm\|e^{t\Delta}\nabla{w}_0\|_{L^{\!\infty}}\nonumber\\
   & =\|\nabla w_0\|_{\dot{B}^{-1}_{\infty,\infty}}\sim\|w_0\|_{\dot{B}^{0}_{\infty,\infty}}\nonumber\\
   &  \le c_1\|w_0\|_{L^\infty},\label{eq:3-8}
    \end{align}
  where the fourth and the fifth inequalities follow from Theorem 1 of
  \cite{Triebel:1983} p.242 and boundedness of Riesz transforms $\frac{\nabla}{\sqrt{-\Delta}}$
  as well as $L^\infty\!\hookrightarrow\!\dot{B}^{0}_{\infty,\infty}$.

 \noindent  \textsc{Part II.}\; {\it Estimates for $\|B(u,w)\|_X$ and $L(u)\|_{Y}$.}
 As for $\|B(u,w)\|_X$, we have
 \begin{align}
    \|B(u,w)\|_X&=\sup_{t>0}\|\!\!\int_0^t\!\! e^{(t-\tau)\Delta} \nabla\cdot(u\nabla w) d\tau \|_{L^1}
                 +\sup_{t>0}\,t\|\!\!\int_0^t\!\! e^{(t-\tau)\Delta} \nabla\cdot(u\nabla w) d\tau \|_{L^\infty}
                  \nonumber\\
      &\le      \sup_{t>0}\int_0^t(t-\tau)^{-\frac{1}{2}}\|u\nabla{w}\|_{L^1}d\tau
           +    \sup_{t>0}t\!\!\int_0^{\frac{t}{2}}(t-\tau)^{-\frac{3}{2}}\|u\nabla{w}\|_{L^1}d\tau\nonumber\\
      &\,\;\; + \sup_{t>0}t\!\!\int_{\frac{t}{2}}^t(t-\tau)^{-\frac{1}{2}}\|u\nabla{w}\|_{L^\infty}d\tau\nonumber\\
      &\le\sup_{t>0}\Big(\int_0^t(t-\tau)^{-\frac{1}{2}}\tau^{-\frac{1}{2}}d\tau+c_2\frac{1}{\!\!\sqrt{t}}\!\int_0^{\frac{t}{2}}\tau^{-\frac{1}{2}}d\tau\Big)\sup_{\tau>0}\tau^{\frac{1}{2}}\|(u\nabla{w})(\tau)\|_{L^1}\nonumber\\
      &\,\;\; + \sup_{t>0}t\!\!\int_{\frac{t}{2}}^t(t-\tau)^{-\frac{1}{2}}\tau^{-\frac{3}{2}}d\tau\sup_{\tau>0}\tau^{\frac{3}{2}}\|(u\nabla{w})(\tau)\|_{L^\infty}\nonumber\\
      &\le c_2\sup_{\tau>0}\,(\|u(\tau)\|_{L^1}+\tau\|u(\tau)\|_{L^\infty})\sup_{\tau>0}\tau^{\frac{1}{2}}\|\nabla{w}(\tau)\|_{L^\infty} \nonumber\\
      &\le {c_2}\|u\|_{X}\|w\|_{Y}.
      \label{eq:3-9}
   \end{align}
 As for $\|L(u)\|_Y$, from definition of $\|\cdot\|_{Y}$, we need to
 estimate
  \begin{align}
      &\sup_{t>0}{t}^{\frac{1}{2}}\|\nabla{L}(u)\|_{L^\infty}  =  \sup_{t>0}t^\frac{1}{2}\|\int_0^te^{-(t-\tau)}e^{(t-\tau)\Delta}\nabla{u}d\tau\|_{L^\infty}\nonumber\\
                             &\le\sup_{t>0}\Big(t^\frac{1}{2}\!\!\int_0^{\frac{t}{2}}(t-\tau)^{-\frac{3}{2}}\|{u}(\tau)\|_{L^1}d\tau +
                                            t^\frac{1}{2}\!\!\int_{\frac{t}{2}}^t(t-\tau)^{-\frac{1}{2}}\tau^{-1}\tau\|{u}(\tau)\|_{L^\infty}d\tau\Big)\nonumber\\
                             &\le{c_3}\sup_{\tau>0}\hskip.02cm(\|u(\tau)\|_{L^1}+\tau\|u(\tau)\|_{L^\infty})\nonumber\\
                             &\le{c}_3\|u\|_X.\label{eq:3-10}
  \end{align}
  Setting $c\!=\max\{c_1,c_2,c_3\}$, combining \eqref{eq:3-4}--\eqref{eq:3-10}, multiplying $B(u,w)$ by $4c$ and
  multiplying $L(u)$ by $\frac{1}{4c}$, we prove
   \eqref{eq:3-5}.
 \end{proof}
 \smallskip
 \textbf{{Proof of Theorem \ref{thm:1.1}}}:\hskip.3cm
 At first, applying Proposition \ref{pro:3.1}, following similar arguments as in the proof of
 Lemma \ref{lem:2.5}, we prove that there exists a unique solution $(u,w)\in \overline{B(0,2A_{10})}\subset{X}\times{Y}$ to system \eqref{eq3-2'} if
 $A_{10}:=\|(e^{t\Delta}u_0,e^{t(\Delta-1)}w_0)\|_{X\times{Y}}\!<\!\!{3}/{32c^2}$.
 Moreover, this solution also satisfies $\Phi(u,w)=(u,w)$. From \eqref{eq:3-7}--\eqref{eq:3-8}, it suffices to
 assume that $\|(u_0,w_0)\|_{L^1\times
 \dot{B}^0_{\infty,\infty}}\!\!<\!{3}/{32c^3}$ since $A_{10}\le c\|(u_0,w_0)\|_{L^1\times\dot{B}^0_{\infty,\infty}}\!\!\!<\!{3}/{32c^2}$.

 Next we show that $w\in
 {C}_w([0,\infty);L^\infty(\mathbb{R}^2))$. From \eqref{eq3-2'} and \eqref{eq:3-6}, we
 have
 \begin{align*}
    \sup_{t>0}\|w\|_{L^\infty}
  &=\sup_{t>0}\|e^{t(\Delta-1)}w_0+\frac{1}{4c}L(u)\|_{L^\infty}\le\|w_0\|_{L^\infty}+\frac{1}{4}\|u\|_X,
 \end{align*}
 where in \eqref{eq:3-10},
 $c_3=\sup_{t>0}(t^{\frac{1}{2}}\!\int_0^{\frac{t}{2}}(t-\tau)^{-\frac{3}{2}}d\tau+t^{\frac{1}{2}}\!\int_{t/2}^{t}(t-\tau)^{-\frac{1}{2}}\tau^{-1}d\tau)$
 and similarly
 \begin{align*}
      \|\frac{1}{4c}L(u)\|_{L^\infty}
  & = \frac{1}{4c}\|\int_0^te^{(t-\tau)(\Delta-1)}ud\tau\|_{L^\infty}\le\frac{1}{4c}\|\int_0^te^{(t-\tau)\Delta}ud\tau\|_{L^\infty}\\
  &\le\frac{1}{4c}\int_0^{\frac{t}{2}}\!(t-\!\tau)^{-\frac{2}{2}(\frac{1}{1}-\frac{1}{\infty})}\|u(\tau)\|_{L^\infty}d\tau+\frac{1}{4c}\int_{\frac{t}{2}}^t
  \tau^{-1}\tau\|u(\tau)\|_{L^\infty}d\tau\nonumber\\
  &\le \frac{c_3}{4c}\|u\|_X\le \frac{1}{4}\|u\|_X
 \end{align*}
 since $\int_0^{\frac{t}{2}}(t-\tau)^{-1}d\tau<
 t^{\frac{1}{2}}\!\int_0^{\frac{t}{2}}(t-\tau)^{-\frac{3}{2}}d\tau$,
 $\int_{\frac{t}{2}}^t\tau^{-1}d\tau<t^{\frac{1}{2}}\!\int_{\frac{t}{2}}^{t}(t-\tau)^{-\frac{1}{2}}\tau^{-1}d\tau$ and $c\ge{c}_3$.
 Moreover, following a dense argument in $L^1(\mathbb{R}^2)$ we can prove the time continuity of $u$.
 Since Schwartz function space is not dense in
 $L^\infty(\mathbb{R}^2)$, we can only obtain the weakly star
 time continuity of solution $w$.

 Finally, performing transformation: $(u,v)=(u,4cw)$, we get the unique solution $(u,v)$ of \eqref{k-s-1}--\eqref{k-s-3}.

\subsection{Proof of Theorem \ref{thm:1.2}}
 In this subsection, we prove global well-posedness of system \eqref{eq3-2'} with initial data
  $(u_0,w_0)\in H^1_b(\mathbb{R}^2)\times H^1(\mathbb{R}^2)$ by making
 use of the Kato's framework, see \cite{Kato:1984471} for instance. At first,
 we recall that $\sigma(t)=t^{\frac{1}{2}}(1+t)^{-\frac{1}{2}}$ and then we set
 \begin{align}\label{eq:3-12}
   \begin{aligned}
    &X =\! \{u\in \mathcal{S}'(\mathbb{R}^2\!\times\! (0,\infty)); \ \sup_{t>0} \|u(\cdot,t)\|_{H^1} + \|\nabla{u}\|_{L^2_tH^1} \!+
    \|u\|_{L^\infty_tL^\infty}   < \infty\;\},\\
    &Y=\! \{w\in\! \mathcal{S}'(\mathbb{R}^2\!\times\!(0,\infty)); \ \sup_{t>0} \|w(\cdot,t)\|_{H^1}\! + \!\|\nabla{w}\|_{L^2_tH^1}\! +
    \|{\sigma\nabla}{w}\|_{L^{\!\infty}_t\!L^{\!\infty}}\!<\! \infty\}.
   \end{aligned}
  \end{align}

 The following Proposition will play a central role in proving
 Theorem \ref{thm:1.2}.
\begin{proposition}\label{pro:3.2}
  For any initial data $({u}_0,w_0)\in H^1_b(\mathbb{R}^2)\times H^1(\mathbb{R}^2)$, there exists positive constant $c$ such that
  \begin{align}\label{eq:3-13}
   &\|4cB(u,w)\|_{X}\le{4c^2}\|u\|_{X}\|w\|_{Y}\le{c^2}\|(u,w)\|_{X\times{Y}}^2,\quad \|\frac{1}{4c}L(u)\|_Y\le \frac{1}{4}\|u\|_X
  \end{align}
   and\; $\|(e^{t\Delta}u_0,e^{t(\Delta-1)}w_0)\|_{X\times
  Y}\le{c}\|(u_0,v_0)\|_{H^1_b\times{H}^1}$.
 \end{proposition}
 \begin{proof}
 We divide the whole proof into two parts concerning with
 $(e^{t\Delta}u_0,\, e^{-t}e^{t\Delta}w_0)$ and  $(-4cB(u,w),\, \frac{1}{4c}L(u))$, respectively.

 \noindent \textsc{Part I.}\; {\it Estimates for $\|e^{t\Delta}u_0\|_X$ and $\|e^{-t}e^{t\Delta}w_0\|_Y$.}
   As for $\|e^{t\Delta}u_0\|_X$, noticing that $e^{-ct|\xi|^2}\in L^\infty_tL^\infty_\xi$ and
   $e^{-ct|\xi|^2}\xi\in L^\infty_\xi L^2_t$, then by applying Lemma \ref{lem:2.1} and
   Young's inequality, we
   have
 \begin{align}
   \|e^{t\Delta}u_0\|_X
       & = \|e^{t\Delta}u_0\|_{L^\infty_tH^1}+\|e^{t\Delta}\nabla{u_0}\|_{L^2_tH^1}+\|e^{t\Delta}u_0\|_{L^\infty_tL^\infty}\nonumber\\
       &\le\|u_0\|_{H^1}+\|u_0\|_{H^1}+\|u_0\|_{L^\infty}\nonumber\\
       &\le c\|u_0\|_{H^1_b}.\label{eq:3-14}
 \end{align}
 Similarly, we have
 \begin{align}
   \|e^{-t}e^{t\Delta}w_0\|_{L^\infty_tH^1}+\|e^{-t}e^{t\Delta}\nabla{w}_0\|_{L^2_tH^1} \le {c}\|w_0\|_{H^1}.\label{eq:3-15}
 \end{align}
 Recall that $\sigma(t)=t^{\frac{1}{2}}(1+t)^{-\frac{1}{2}}$. Then we
 get
 \begin{align}
  \|{\sigma\nabla} e^{-t}e^{t\Delta}w_0\|_{L^\infty_tL^\infty}
  &=\sup_{t>0}t^{\frac{1}{2}}(1+t)^{-\frac{1}{2}}e^{-t}\|e^{t\Delta}\nabla{w}_0\|_{L^\infty}\nonumber\\
  &\le\sup_{t>0}t^{\frac{1}{2}}\|e^{t\Delta}\nabla{w}_0\|_{L^\infty} \!=\! \|\nabla{w}_0\|_{\dot{B}^{-1}_{\infty,\infty}}\nonumber\\
  &\le{c}\!\;\|\nabla w_0\|_{L^2}\le c\!\;\|w_0\|_{H^1},\label{eq:3-16}
 \end{align}
 where we have used Proposition \ref{def:2.4}, embedding theorems of Besov spaces (cf.
 \cite{Triebel:1983}).

 \noindent\textsc{Part II.}\; {\it Estimates for $\|B(u,w)\|_X$ and $\|L(u)\|_Y$.}
 As for $\|B(u,v)\|_X$, we get
 \begin{align}
 \|B(u,w)\|_X&=\|B(u,w)\|_{L^\infty_tH^1}+\|\nabla B(u,w)\|_{L^2_tH^1}+\|B(u,w)\|_{L^\infty_tL^\infty}\nonumber\\
 & \le \|\int_0^t\!e^{(t-\tau)\Delta} \nabla\cdot(u\nabla w) d\tau\|_{L^\infty_tL^2}
   +\|\int_0^t\!e^{(t-\tau)\Delta}\nabla\nabla\cdot(u\nabla w)
   d\tau\|_{L^\infty_tL^2}\nonumber\\
  &\;\;+\|\!\int_0^t\!e^{(t-\tau)\Delta}\nabla\nabla\cdot(u\nabla w)
   d\tau\|_{L^2_tL^2}\!+\|\!\int_0^t\!e^{(t-\tau)\Delta}\Delta\!\nabla\cdot(u\nabla w)
   d\tau\|_{L^2_tL^2}\nonumber\\
  &\;\;+ \|\!\int_0^t\!e^{(t-\tau)\Delta}\nabla\cdot(u\nabla w)
   d\tau\|_{L^\infty_tL^\infty}\nonumber\\
 &:=I_1+I_2+I_3+I_4+I_5,\nonumber\end{align}
 where by applying Lemma \ref{lem:2.3}, we have
 \begin{align}\label{eq:3-17}
 &I_1 \le c\, \|u\nabla w\|_{L^2_tL^2}\le{c}\,\|u\|_{L^\infty_tL^\infty}\|\nabla{w}\|_{L^2_tL^2},\\
 &I_2\le c\, \|u \nabla w\|_{L^\infty_tL^2}\le{c}\,\|u\|_{L^\infty_tL^\infty}\|\nabla w\|_{L^\infty_tL^2}
 \end{align}
 and by applying Lemma \ref{lem:2-2}, H\"older's inequality and interpolation theorem, we have
 \begin{align}
  &I_3\le{c}\,\|u\nabla{w}\|_{L^2_tL^2}\le{c}\,\|u\|_{L^\infty_tL^\infty}\|\nabla{w}\|_{L^2_tL^2},\label{eq:3-18}.\end{align}
 Similarly, we have
 \begin{align}
  &I_4\le{c}\|\nabla\cdot({u\nabla{w}})\|_{L^2_tL^2}
  \le{c}\hskip.04cm(\|\nabla{u}\cdot\nabla{w}\|_{L^2_tL^2} + \|u\Delta{w}\|_{L^2_tL^2})\nonumber\\
  &\;\;\;\le{c}\hskip.02cm(\|\nabla{u}\|_{L^4_tL^4}\|\nabla{w}\|_{L^4_tL^4}+\|u\|_{L^\infty_tL^\infty}\|\nabla{w}\|_{L^2_t\dot{H}^1})\nonumber\\
  &\;\;\;\le{c}\hskip.02cm(\|{u}\|_{L^{\infty}_tH^1}^{\frac{1}{2}}\|\nabla{u}\|_{L^2_t{H}^{1}}^{\frac{1}{2}}\|{w}\|_{L^{\infty}_tH^1}^{\frac{1}{2}}
   \|\nabla{w}\|_{L^2_t{H}^{1}}^{\frac{1}{2}} + \|u\|_{L^{\infty}_t\!L^{\!\infty}}\|\nabla{w}\|_{L^2_t{H}^1}),\label{eq:3-18'}
                       \end{align}
 where $\|f\|_{L^4_tL^4}^2\le
 c\hskip.01cm\|f\|_{L^{4}_t\dot{H}^{\frac{1}{2}}}^2\le{c}\hskip.01cm\|f\|_{L^{\infty}_tL^2}\|\nabla{f}\|_{L^2_t{H}^{1}}$.
  As for $I_5$, by splitting the time interval, we obtain that if $t>2$, then $1<t-1<\tau<t$, $1<\frac{1}{{\sigma(\tau)}}=
  \frac{\sqrt{1+\tau}}{\sqrt{\tau}}<2$ and
 \begin{align}
   I_5& = \|\int_0^t\!e^{(t-\tau)\Delta}\nabla\cdot(u\nabla w) d\tau\|_{L^\infty_tL^\infty}\nonumber\\
      & \le \|\!\int_0^{t-1}\!e^{(t-\tau)\Delta}\nabla\!\cdot(u\nabla w) d\tau\|_{L^{\!\infty}_tL^{\!\infty}}
       +\|\!\int_{t-1}^t\!e^{(t-\tau)\Delta}\nabla\cdot(u\nabla w) d\tau\|_{L^{\!\infty}_tL^{\!\infty}}\nonumber\\
      &\le c\int_0^{t-1}\!\!(t-\tau)^{-\frac{3}{2}}\|u\nabla{w}\|_{L^{1}}d\tau
        + c\int_{t-1}^t(t-\tau)^{-\frac{1}{2}}\|u\nabla w\|_{L^\infty}d\tau\nonumber\\
      &\le c\,(\|u\nabla w\|_{L^\infty_tL^1}+\|u\nabla w\|_{L^\infty_tL^\infty})\nonumber\\
      &\le c\,(\|u\|_{L^\infty_tL^2}\|\nabla{w}\|_{L^\infty_tL^2}+\|u\|_{L^\infty_tL^\infty}\|{\sigma\nabla}{w}\|_{L^\infty_tL^\infty});\label{eq:3-19}
   \end{align}
   Else if $0<t\le 2$ and $0<\tau<t$, then we get
   $\tau^{\frac{1}{2}}/2<{\sigma(\tau)}<\tau^{\frac{1}{2}}$ and
 \begin{align}
   I_5& = \|\int_0^t\!e^{(t-\tau)\Delta}\nabla\cdot(u\nabla w) d\tau\|_{L^\infty_tL^\infty}\nonumber\\
      & \le \|\int_0^{t}(t-\tau)^{-\frac{1}{2}}\frac{1}{{\sigma(\tau)}} {\sigma(\tau)}\|(u\nabla w)(\tau)\|_{L^\infty}d\tau\nonumber\\
      &\le c\!\int_0^t (t-\tau)^{-\frac{1}{2}}\tau^{-\frac{1}{2}} \|u\|_{L^\infty_tL^\infty}\|{\sigma\nabla}{w}\|_{L^\infty_tL^\infty}\nonumber\\
      &\le c\,\|u\|_{L^\infty_tL^\infty}\|{\sigma\nabla}{w}\|_{L^\infty_tL^\infty}.\label{eq:3-20}
 \end{align}
 In order to estimate $\|L(u)\|_Y$, we have
 \begin{align}\label{eq:3-21}
   \|L(u)\|_Y & = \|L(u)\|_{L^\infty_tH^1}+\|\nabla{L}(u)\|_{L^2_tH^1}+\|{\sigma\nabla}{w}\|_{L^\infty_tL^\infty} \nonumber\\
   :&= I_6+I_7+I_8,
 \end{align}
 where by applying Lemma \ref{lem:2.3} \eqref{eq:2.7} to $I_5$, we have
 \begin{align}
  I_6&=\|L(u)\|_{L^\infty_tH^1}=\|\int_0^te^{-t+\tau}e^{(t-\tau)\Delta}u\|_{L^\infty_tH^1}\nonumber\\
     &\le \sup_{t>0}\int_0^te^{-t+\tau}\|u\|_{H^1}d\tau\nonumber\\
     &\le c\|u\|_{L^\infty_tH^1}\label{eq:3-22}
 \end{align}
 and by applying \eqref{eq:2.7} to $\nabla{L(u)}=L(\nabla u)$ and $\Delta{L}(u)=\nabla{L}(\nabla u)$ with $\theta=0$ and $\theta=1$, we have
 \begin{align}\label{eq:3-23}
  I_7&=\|\nabla{L}(u)\|_{L^2_tH^1}\le\|\nabla{L}(u)\|_{L^2_tL^2}+\|\nabla\nabla {L}(u)\|_{L^2_tL^2}\nonumber\\
     &\le\|{L}(\nabla u)\|_{L^2_tL^2}+\|\nabla{L}(\nabla u)\|_{L^2_tL^2}\nonumber\\
     &\le c\,\|\nabla u\|_{L^2_tL^2}.
 \end{align}
 It remains to estimate $I_8=\|{\sigma\nabla}{w}\|_{L^\infty_tL^\infty}$.
 Recalling the definition of $w$, we get
 \begin{align}
  I_8&=\sup_{t>0}t^{\frac{1}{2}}(1+t)^{-\frac{1}{2}}\|\int_0^te^{-t+\tau}e^{(t-\tau)\Delta}\nabla{u}d\tau\|_{L^\infty}\nonumber\\
     &\le c\sup_{t>0}\int_0^t e^{-t+\tau}(t-\tau)^{-\frac{1}{2}}d\tau
       \|u\|_{L^\infty_tL^\infty}\nonumber\\
     &\le c\,\|u\|_{L^\infty_tL^\infty}.\label{eq:3-24}
 \end{align}
   Combining  \eqref{eq:3-14}--\eqref{eq:3-24}, we prove \eqref{eq:3-13} and hence finish the proof.
 \end{proof}
 \smallskip
\noindent \textbf{{Proof of Theorem \ref{thm:1.1}}}:\hskip.3cm
Applying Proposition \ref{pro:3.2}, following
 similar arguments as in Lemma \ref{lem:2.5} we can prove Theorem \ref{thm:1.2} and hence we omit the details.

\medskip\medskip
 \noindent {\bf Acknowledgmens:}
 Chao Deng is supported by PAPD of Jiangsu Higher Education Institutions and Jiangsu Normal University under Grant No.\! 9212112101;
 He is also supported by the China Natural Science Foundation under Grant No.\,11171357 and No.\,11271166;
 He would like to express his gratitude to the Applied Mathematics Department of Colorado
 University at Boulder for its hospitality. Congming Li is partially
 supported by the NSF grants DMS-0908097.

\vskip1cm {
 {\small Chao Deng,}\\
 {\small  Department of Mathematics, Jiangsu Normal University}\\
 {\small {\it Email: deng315@yahoo.com.cn}}\\
 {\small and }\\
 {\small Department of Applied Mathematics, Colorado University at Boulder.}\\
   \\ \\
 {\small Congming Li,}\\
 {\small Department of Applied Mathematics, Colorado University at Boulder}\\
   {\small  {\it Email: cli@colorado.edu}}\\
   {\small and}\\
 {\small Department of Mathematics and MOE-LSC, Shanghai Jiao Tong
 University.}\\

 }
 \date{}

\end{document}